\documentclass[11pt]{amsart}

\usepackage{amssymb,amsmath,epsfig,graphics}

\numberwithin{equation}{section}
\newcommand{\beq}{\begin{equation}}
\newcommand{\eeq}{\end{equation}}
\newcommand{\bea}{\begin{eqnarray}}
\newcommand{\eea}{\end{eqnarray}}
\newcommand{\beas}{\begin{eqnarray*}}
\newcommand{\eeas}{\end{eqnarray*}}

\newtheorem{theorem}{Theorem}[section]
\newtheorem{definition}[theorem]{Definition}
\newtheorem{proposition}[theorem]{Proposition}
\newtheorem{corollary}[theorem]{Corollary}

\newtheorem{remark}[theorem]{Remark}
\newtheorem{example}[theorem]{Example}
\newtheorem{examples}[theorem]{Examples}
\newtheorem{foo}[theorem]{Remarks}

\newcommand{\bG}{\mathbb G}

\newcommand{\p}{\partial}

\newcommand{\bM}{\mathbb M}

\newcommand{\Ho}{\mathcal H}

\newcommand{\M}{\mathbb M}

\newcommand{\R}{\mathbb R}

\newcommand{\ve}{\varepsilon}

\title[A sub-Riemannian curvature-dimension inequality, etc.]{A sub-Riemannian curvature-dimension inequality, volume doubling property and the Poincar\'e inequality}

\author{Fabrice Baudoin}
\address{Department of Mathematics\\Purdue University \\
West Lafayette, IN 47907} \email[Fabrice Baudoin]{fbaudoin@math.purdue.edu}
\thanks{First author supported in part by
NSF Grant DMS 0907326}

\author{Michel Bonnefont}
\address{Institut de Math\'ematiques de Toulouse
\\Universit\'e Paul Sabatier \\
31062 Toulouse Cedex 9, France} \email[Michel Bonnefont]{michel.bonnefont@math.univ-toulouse.fr}

\author{Nicola Garofalo}
\address{Department of Mathematics
\\The Ohio State University\\ 
100 Math Tower\\ 231 West 18th Avenue\\ Columbus, OH 43210-1174} 
\email[Nicola
Garofalo]{rembrandt54@gmail.com}
\thanks{Third author supported in part by NSF Grant DMS-1001317}

\begin{document}
\maketitle

\begin{abstract}
Let $\M$ be a smooth connected manifold endowed with a smooth measure $\mu$ and a smooth locally subelliptic diffusion operator $L$ satisfying $L1=0$, and which is symmetric with respect to $\mu$. We show that if $L$ satisfies, with a non negative curvature parameter,  the generalized curvature inequality introduced by the first and third named authors in \cite{BG}, then the following properties hold:

\

\begin{itemize}
\item The volume doubling property;
\item The Poincar\'e inequality;
\item The parabolic Harnack inequality.
\end{itemize}

\

The key ingredient is the study of dimension dependent reverse log-Sobolev inequalities for the heat semigroup and corresponding non-linear reverse Harnack type inequalities. Our results apply in particular to all Sasakian manifolds whose horizontal Webster-Tanaka-Ricci curvature is nonnegative, all Carnot groups of step two, and to wide subclasses of principal bundles over Riemannian manifolds whose Ricci curvature is nonnegative. 
\end{abstract}

\tableofcontents

\section{Introduction}\label{S:intro}

A fundamental property of a measure metric space $(X,d,\mu)$ is the so-called \emph{doubling condition} stating that for every $x\in X$ and every $r>0$ one has
\begin{equation}\label{gdc}
\mu(B(x,2r))\le C_d\ \mu(B(x,r)),
\end{equation}
for some constant $C_d>0$, where $B(x,r) = \{y\in X\mid d(y,x)<r\}$. As it is well-known, such property is central for the validity of covering theorems of Vitali-Wiener type, maximal function estimates, and it represents one of the central ingredients in the development of analysis and geometry on metric measure spaces, see for instance \cite{Fe}, \cite{FS}, \cite{CW},  \cite{HK}, \cite{He}, \cite{Ha}, \cite{AT}.
Another fundamental property is  the \emph{Poincar\'e inequality} which claims the existence of  constants $C_p>0$ and $a\ge 1$  such that for every Lipschitz function $f$ on $B(x,ar)$ one has
\begin{equation}\label{gpi}
\int_{B(x,r)} |f - f_B|^2 d\mu \le C_p r^2 \int_{B(x,ar)} g^2 d\mu,
\end{equation}
where we have let $f_B = \mu(B)^{-1} \int_B f d\mu$, with $B = B(x,r)$. In the right-hand side of \eqref{gpi} the function $g$ denotes an \emph{upper gradient} for $f$ (see \cite{Che}  and \cite{HeK} for a discussion of upper gradients). 

One basic instance of a measure metric space supporting \eqref{gdc} and \eqref{gpi} is a complete $n$-dimensional Riemannian manifold $\M$ with nonnegative Ricci tensor. In such case  \eqref{gdc} follows with $C_d = 2^n$ from the Bishop-Gromov comparison theorem (see e.g. Theorem 3.10 in \cite{Cha}), whereas \eqref{gpi} was proved by Buser in \cite{Bu}, with $ a=1$ and $g = |\nabla f|$.

Beyond the classical Riemannian case two situations of considerable analytic and geometric interest are CR and sub-Riemannian manifolds. For these classes global inequalities such as \eqref{gdc} and \eqref{gpi} are mostly \emph{terra incognita}. 
The purpose of the present paper is taking a first step in filling this gap in the class of sub-Riemannian manifolds that satisfy the generalized curvature dimension inequality introduced in \cite{BG}. Our main result, Theorem \ref{T:main} below, constitutes a sub-Riemannian counterpart of the case in which Ricci $\ge 0$ (for this aspect, see e.g. Theorem \ref{T:sasakian} below). 

To introduce our results, we recall that a $n$-dimensional Riemannian manifold $\M$ with Laplacian $\Delta$ is said to satisfy the \emph{curvature-dimension inequality} CD$(\rho_1,n)$ if there exists $\rho_1\in \R$ such that for every $f\in C^\infty(\M)$ one has
\begin{equation}\label{cdr}
\Gamma_2(f) \ge \frac 1n (\Delta f)^2 + \rho_1 |\nabla f|^2,
\end{equation}
where 
\[
\Gamma_2(f) =  \frac 12 \big(\Delta |\nabla f|^2 - 2 <\nabla f,\nabla(\Delta f)>\big).
\]
This notion was introduced by Bakry and Emery in \cite{Bakry-Emery}, and it was further developed in \cite{bakry-tata}, \cite{St1}, \cite{St2}, \cite{LV}. What is remarkable about the curvature-dimension inequality \eqref{cdr} is that it holds on a Riemannian manifold $\M$ if and only if Ric$ \ge \rho_1$.
It follows that such notion could be taken as an alternative characterization of Ricci lower bounds. 

This point of view was recently taken up by the first and third named authors in \cite{BG}, where a new sub-Riemannian curvature-dimension inequality was introduced. Such new inequality was shown to constitute a very robust tool for developing a Li-Yau type program in some large classes of sub-Riemannian manifolds. In the present paper we develop our program even further, and in a different direction, by proving that the generalized curvature-dimension inequality introduced in \cite{BG} can be successfully used to establish global inequalities such as \eqref{gdc} and \eqref{gpi} above.

To state our main results we now introduce the relevant framework.  We consider measure metric spaces $(\M,d,\mu)$, where $\M$ is $C^\infty$ connected manifold endowed with a $C^\infty$ measure $\mu$, and $d$ is a metric canonically associated with a $C^\infty$ second-order diffusion operator $L$ on $\M$ with real coefficients. We assume that $L$ is locally subelliptic on $\M$ in the sense of \cite{FP}, and that moreover: 
\begin{itemize}
\item[(i)] $L1=0$; 
\item[(ii)]
$\int_\bM f L g d\mu=\int_\bM g Lf d\mu$;
\item[(iii)] $\int_\bM f L f d\mu \le 0$,
\end{itemize}
for every $f , g \in C^ \infty_0(\bM)$. The following distance is canonically associated with the operator $L$:
\begin{equation}\label{di}
d(x,y)=\sup \left\{ |f(x) -f(y) | \mid f \in  C^\infty(\bM) , \| \Gamma(f) \|_\infty \le 1 \right\},\ \ \  \ x,y \in \bM,
\end{equation}
where for a function $g$ on $\bM$ we have let $||g||_\infty = \underset{\bM}{\text{ess} \sup} |g|$.

Given the manifold $\M$ and the diffusion operator $L$,  similarly to \cite{Bakry-Emery} we consider the quadratic functional $\Gamma(f) = \Gamma(f,f)$, where
\begin{equation}\label{gamma}
\Gamma(f,g) =\frac{1}{2}(L(fg)-fLg-gLf), \quad f,g \in C^\infty(\bM),
\end{equation}
is known as \textit{le carr\'e du champ}. One should in fact think of $\Gamma(f)$ as the square of the length of the gradient of $f$ along the so-called horizontal directions. We remark that $\Gamma$ depends only on the diffusion operator $L$, and in this sense it is canonical. Notice that $\Gamma(f) \ge 0$ and that $\Gamma(1) = 0$. 

Unfortunately, in sub-Riemannian geometry the canonical bilinear form $\Gamma$ does not suffice to develop the Li-Yau program. To circumvent this obstruction, we further suppose that $\bM$ is equipped with a symmetric, first-order differential bilinear form $\Gamma^Z:C^\infty(\bM)\times C^\infty(\bM) \to C^\infty(\bM)$, satisfying
\[
\Gamma^Z(fg,h) = f\Gamma^Z(g,h) + g \Gamma^Z(f,h).
\]
We make the assumption that $\Gamma^Z(f) = \Gamma^Z(f,f) \ge 0$ (one should notice that $\Gamma^Z(1) = 0$). Roughly speaking, in a sub-Riemannian manifold $\Gamma^Z(f)$ represents the square of the length of the gradient of $f$ in the directions of the commutators. We emphasize that, in the above general formulation, the bilinear form $\Gamma^Z$ is not canonical since, unlike the form $\Gamma$, a priori it has no direct correlation to the diffusion operator $L$. The reader should however find reassuring that, in all the concrete geometric examples encompassed by the present paper, the choice of the form $\Gamma^Z$ can be shown to be, in fact, canonical.  

To clarify this important point we pause for a moment to discuss a basic class of three-dimensional models which have been analyzed in Section 2 in \cite{BG}.
Given a $\rho_1\in \R$ we consider a Lie group $\bG(\rho_1)$ whose Lie algebra $\mathfrak g$ admits a basis of generators $X, Y, Z$ satisfying the commutation relations
\begin{equation}\label{comm}
[X,Y] = Z,\ \ \ [X,Z] = - \rho_1 Y,\ \ \ [Y,Z] = \rho_1 X.
\end{equation}
The group $\bG(\rho_1)$ can be endowed with a natural CR structure $\theta$ with respect to which the Reeb vector field is given by $-Z$. A sub-Laplacian  on $\bG(\rho_1)$ with respect to such structure is thus given by $L = X^2 + Y^2$. The pseudo-hermitian Tanaka-Webster torsion of $\bG(\rho_1)$ vanishes, and thus $(\bG(\rho_1),\theta)$ is a Sasakian manifold. In the smooth manifold $\bM = \bG(\rho_1)$ with sub-Laplacian $L$ we introduce the differential forms $\Gamma$ and $\Gamma^Z$  defined by 
\[
\Gamma(f,g) = Xf Xg + Yf Yg,\ \ \ \ \ \ \ \Gamma^Z(f,g) = Zf Zg.
\]
It is worth observing that, since as we have said $-Z$ is the Reeb vector field of the CR structure $\theta$, then the above choice of $\Gamma^Z$ is canonical. It is also worth remarking at this point that for the CR manifold $(\bG(\rho_1),\theta)$ the Tanaka-Webster horizontal sectional curvature is constant and equals $\rho_1$. For instance, when $\bG$ is the $3$-dimensional Heisenberg group $\mathbb H^1$, with real coordinates $(x,y,t)$, and generators of the Lie algebra $X = \p_x-\frac y2 \p_t$, $Y=\p_y + \frac x2 \p_t$, $Z = \p_t$, then \eqref{comm} holds with $\rho_1 = 0$. In \cite{BG} two other special instances of the model CR manifold $\bG(\rho_1)$ were discussed in detail, namely $\mathbb{SU}(2)$, and $\mathbb{SL}(2,\R)$, corresponding, respectively, to the cases $\rho_1 = 1$ and $\rho_1 = -1$.

Given the first-order bilinear forms $\Gamma$ and $\Gamma^Z$ on $\bM$, we now introduce the following second-order differential forms:
\begin{equation}\label{gamma2}
\Gamma_{2}(f,g) = \frac{1}{2}\big[L\Gamma(f,g) - \Gamma(f,
Lg)-\Gamma (g,Lf)\big],
\end{equation}
\begin{equation}\label{gamma2Z}
\Gamma^Z_{2}(f,g) = \frac{1}{2}\big[L\Gamma^Z (f,g) - \Gamma^Z(f,
Lg)-\Gamma^Z (g,Lf)\big].
\end{equation}
Observe that if $\Gamma^Z\equiv 0$, then $\Gamma^Z_2 \equiv 0$ as well. As for $\Gamma$ and $\Gamma^Z$, we will use the notations  $\Gamma_2(f) = \Gamma_2(f,f)$, $\Gamma_2^Z(f) = \Gamma^Z_2(f,f)$.

The next definition, which we are taking from \cite{BG},  is the central character of this paper.

\begin{definition}[Generalized curvature-dimension inequality]\label{D:cdi}
Let  $\rho_1 \in \mathbb{R}$, $\rho_2 >0$, $\kappa \ge 0$, and $m> 0$. We say that $\bM$ satisfies the \emph{generalized curvature-dimension inequality} \emph{CD}$(\rho_1,\rho_2,\kappa,m)$ if  the inequality 
\begin{equation}\label{cdi}
\Gamma_2(f) +\nu \Gamma_2^Z(f) \ge \frac{1}{m} (Lf)^2 +\left( \rho_1 -\frac{\kappa}{\nu} \right) \Gamma(f) +\rho_2 \Gamma^Z(f)
\end{equation}
 holds for every  $f\in C^\infty(\bM)$ and every $\nu>0$.
\end{definition}

\begin{remark}
We  observe explicitly that by optimizing with respect to the parameter $\nu$ in \eqref{cdi}, we obtain the following equivalent form of the generalized curvature-dimension inequality \emph{CD}$(\rho_1,\rho_2,\kappa,m)$
\begin{equation}\label{cdii}
\Gamma_{2}(f)+2\sqrt{ \kappa \Gamma(f)  \Gamma^Z_{2}(f) }\ge\frac{1}{m}(Lf)^2+   \rho_1   \Gamma (f) + \rho_2 \Gamma^Z (f).
\end{equation}
\end{remark}

Let us also observe right-away that if $\rho_1'\ge \rho_1$, then $CD(\rho_1',\rho_2,\kappa,m) \Longrightarrow$\ CD$(\rho_1,\rho_2,\kappa,m)$. It should be clear to the reader that Definition \ref{D:cdi} constitutes a generalization of the \emph{curvature-dimension inequality} CD$(\rho_1,n)$ from Riemannian geometry \eqref{cdr}. In fact, to see that \eqref{cdi}, or equivalently \eqref{cdii}, contains \eqref{cdr} it is enough to take $L = \Delta$, $\Gamma^Z \equiv 0$, and $m = n$, and notice that \eqref{gamma} gives $\Gamma(f) = |\nabla f|^2$ (also note that in this context the distance \eqref{di} is simply the Riemannian distance on $\M$).  

In Definition \ref{D:cdi} the parameter $\rho_1$ plays a special role. For the results in this paper such parameter represents the lower bound on a sub-Riemannian generalization of the Ricci tensor, see Section 2 in \cite{BG}. Thus, $\rho_1 \ge 0$ is, in our framework, the counterpart of  the Riemannian Ric\ $\ge 0$. For this reason, when in this paper we say that $\M$ satisfies the curvature dimension inequality CD$(\rho_1,\rho_2,\kappa,m)$ with $\rho_1\ge 0$, we will routinely avoid repeating at each occurrence the sentence ``for some $\rho_2>0$, $\kappa\ge 0$ and $m>0$''. To illustrate this point we return for a moment to the Sasakian model space introduced in \eqref{comm} above, and recall the following proposition established in Section 2.2 of \cite{BG}. 

\begin{proposition}\label{P:Grho}
The sub-Laplacian $L$ on the Lie group $\mathbb{G}(\rho_1)$ satisfies the  generalized curvature-dimension inequality \emph{CD}$(\rho_1,\frac{1}{2},1,2)$.
\end{proposition}

The essential new aspect of the generalized curvature-dimension inequality CD$(\rho_1,\rho_2,\kappa,m)$ with respect to the Riemannian inequality CD$(\rho_1,n)$ in \eqref{cdr} is the presence of the a priori non-intrinsic bilinear forms $\Gamma^Z$ and $\Gamma^Z_2$. 
As in \cite{BG}, to be able to handle these non-intrinsic forms we will assume throughout the paper the following
hypothesis (H.1), (H.2) and (H.3). Even if they will not be mentioned explicitly in every individual result, the reader should assume that they are always in force:
\begin{itemize}
\item[(H.1)] There exists an increasing
sequence $h_k\in C^\infty_0(\bM)$   such that $h_k\nearrow 1$ on
$\bM$, and \[
||\Gamma (h_k)||_{\infty} +||\Gamma^Z (h_k)||_{\infty}  \to 0,\ \ \text{as} \ k\to \infty.
\]
\item[(H.2)]  
For any $f \in C^\infty(\bM)$ one has
\[
\Gamma(f, \Gamma^Z(f))=\Gamma^Z( f, \Gamma(f)).
\]
\item[(H.3)]  The heat semigroup generated by $L$, which will denoted $P_t$ throughout the paper, is stochastically complete that is, for $t \ge 0$, $P_t 1=1$ and   for every $f \in C_0^\infty(\bM)$ and $T \ge 0$, one has 
\[
\sup_{t \in [0,T]} \| \Gamma(P_t f)  \|_{ \infty}+\| \Gamma^Z(P_t f) \|_{ \infty} < +\infty.
\]

\end{itemize}

In addition to (H.1)-(H.3), throughout the paper we also assume that:
\begin{itemize}
\item[(H.4)] Given any two points $x, y\in \M$, there exist a subunit curve (in the sense of \cite{FP}), joining them.
\item[(H.5)] The metric space $(M,d)$ is complete. 
\end{itemize}

\begin{remark}\label{R:hyp}
We stress that, 
 when $\bM$ is a complete Riemannian manifold satisfying for some $\rho_1\in \R$ the Bakry-Emery curvature-dimension inequality CD$(\rho_1,n)$ in \eqref{cdr} above, with $L = \Delta$, then (H.1)-(H.5) are fulfilled. In fact, (H.2) is trivially satisfied since we can take $\Gamma^Z \equiv 0$, whereas (H.1) follows from (and it is in fact equivalent to) the completeness of $(\M,d)$. Sub-unit curves are geodesics and thus (H.4) holds. The gradient estimate for $P_t$ in (H.3) follows from the work of Bakry in \cite{BCRAS}, whereas the stochastic completeness $P_t 1 = 1$ was proved by Yau in \cite{Yau3} under the assumption \emph{Ric} $\ge \rho_1$, which was shown by Bakry to be equivalent to \emph{CD}$(\rho_1,n)$, see Proposition 6.2 in \cite{bakry-stflour}. 
\end{remark}

We note that in the geometric examples encompassed by the framework of this paper (for a detailed discussion of these examples the reader should consult Section 2 in \cite{BG}), (H.1) is equivalent to assuming that $(\M,d)$ be a complete metric space, i.e., (H.5).
The assumption (H.4) is for instance fulfilled when the operator $L$ satisfies the finite rank condition of the Chow-Rashevsky theorem, see Section \ref{SS:frame} below for a more detailed discussion. When (H.4) holds, definition \eqref{di} above provides a true distance, and the metric space $(\M,d)$ is a length-space in the sense of Gromov.  
The hypothesis (H.2) is of a geometric nature. For instance, all CR manifolds which are Sasakian satisfy it. 
It is important to mention that the hypothesis (H.3) has been shown in \cite{BG}  to be a consequence of the curvature dimension inequality  \emph{CD}$(\rho_1,\rho_2,\kappa,m)$ in the large class of sub-Riemanniann manifolds with transverse symmetries of Yang-Mills type. Such class encompasses Riemannian structures, CR Sasakian structures, and Carnot groups of step two. Therefore, the assumption (H.3) should not be seen as restrictive if we assume that the curvature dimension inequality is satisfied. We can also observe that the stochastic completeness of $P_t$ is intimately related to the volume growth of large metric balls and has been extensively studied in the litterature (see for instance \cite{Gri2} and \cite{St1})

The following is the central result of the present paper.

\begin{theorem}\label{T:main}
Suppose that the generalized curvature-dimension inequality hold for some $\rho_1 \ge 0$. Then, there exist constants $C_d, C_p>0$, depending only on $\rho_1, \rho_2, \kappa, m$, for which one has for every $x\in \M$ and every $r>0$:
\begin{equation}\label{dcsr}
\mu(B(x,2r)) \le C_d\ \mu(B(x,r));
\end{equation}
\begin{equation}\label{pisr}
\int_{B(x,r)} |f - f_B|^2 d\mu \le C_p r^2 \int_{B(x,r)} \Gamma(f)  d\mu,
\end{equation}
for every $f\in C^1(\overline B(x,r))$.
\end{theorem}

We note explicitly that the possibility of having the same ball in both sides of \eqref{pisr} is due to the above mentioned fact that $(\M,d)$ is a length-space. This follows from the assumption \eqref{S} below (which guarantees that $(\M,d)$ is a Carnot-Carath\'eodory space), and from Proposition 2.2 in \cite{DGN} (which states that every Carnot-Carath\'eodory space is a length-space). Once we know that  $(\M,d)$ is a length-space, we can follow the arguments in D. Jerison's paper \cite{J} on the local Poincar\'e inequality to replace the integral on a larger ball in the right-hand side of \eqref{pisr} with an integral on the same ball $B(x,r)$ as in the left-hand side. To put Theorem \ref{T:main} in the proper perspective we note that, besides the already cited case of a complete Riemannian manifold having Ric$\ \ge 0$, the only genuinely sub-Riemannian manifolds in which \eqref{dcsr} and \eqref{pisr} are presently known to simultaneously hold are stratified nilpotent Lie groups, aka Carnot groups, and, more in general, groups with polynomial growth. In Carnot groups the doubling condition \eqref{dcsr} follows from a simple rescaling argument based on the non-isotropic group dilations, from the group left-translations and form the fact that the push-forward to the group of the Lebesgue measure on the Lie algebra is a bi-invariant Haar measure. For more general Lie groups with polynomial growth Varopoulos gave an elementary proof of the Poincar\'e inequality \eqref{pisr} in \cite{V}. One should also see Theorem 8.2.9 in \cite{VSC}, in which the authors establish two-sided global Gaussian bounds in a Lie group with polynomial growth.  As it is well known, such bounds are equivalent to the doubling condition and the Poincar\'e inequality.

It is worth mentioning at this point that, when $L$ is a sum of square of vector fields like in H\"ormander's work on hypoellipticity \cite{H}, then a local (both in $x\in X$ and $r>0$) doubling condition was proved in the paper \cite{NSW}. In this same framework, a local version of the Poincar\'e inequality was proved by D. Jerison in \cite{J}. But no geometry is of course involved in these fundamental local results. The novelty of our work is in the global character of the estimates \eqref{dcsr} and \eqref{pisr}.

In order to elucidate some of the new geometric settings covered by this paper, we recall that one of the main motivations for the work \cite{NSW} was understanding boundary value problems coming from several complex variables and CR geometry. In connection with CR manifolds we mention that in \cite{BG} the first and third named authors proved the following result.

\begin{theorem}\label{T:sasakiani}
Let $(\bM,\theta)$ be a complete \emph{CR} manifold  with real dimension $2n+1$ and vanishing Tanaka-Webster torsion, i.e., a Sasakian manifold. If 
for every $x\in \bM$ the Tanaka-Webster Ricci tensor satisfies the bound  
\[
\emph{Ric}_x(v,v)\ \ge \rho_1|v|^2,
\]
for every horizontal vector $v\in \mathcal H_x$,
then the curvature-dimension inequality \emph{CD}$(\rho_1,\frac{n}{2},1,2n)$ holds.
\end{theorem}

By combining Theorem \ref{T:main} with Theorem \ref{T:sasakiani} we obtain the following result which provides a large class of new geometric examples which are encompassed by our results, and which could not be previously covered by the existing works.

\begin{theorem}\label{T:sasakian}
Let $\M$ be a Sasakian manifold of real dimension $2n+1$. If 
for every $x\in \M$ the Tanaka-Webster Ricci tensor satisfies the bound  \emph{Ric}$_x \ge 0$, when restricted to the horizontal subbundle $\mathcal H_x$,
then there exist constants $C_d, C_p>0$, depending only on $n$, for which one has for every $x\in \M$ and every $r>0$:
\begin{equation}\label{dcsas}
\mu(B(x,2r)) \le C_d\ \mu(B(x,r));
\end{equation}
\begin{equation}\label{pisas}
\int_{B(x,r)} |f - f_B|^2 d\mu \le C_p r^2 \int_{B(x,r)} |\nabla_H f|^2 d\mu.
\end{equation}
\end{theorem}

In \eqref{pisas} we have denoted with $\nabla_Hf$ the horizontal gradient of a function $f\in C^1(\overline B(x,r))$. Concerning Theorem \ref{T:sasakian} we mention that in their recent work \cite{AL} Agrachev and Lee, with a completely different approach from us, have obtained \eqref{dcsas} and \eqref{pisas} for three-dimensional Sasakian manifolds.

Once Theorem \ref{T:main} is available, then from the work of Grigor'yan \cite{Gri} and Saloff-Coste  \cite{SC} (see also \cite{FS},  \cite{KS3}, \cite{St1}, \cite{St2}, \cite{St3}) it is well-known that, in a very general Markov setting, the conjunction of \eqref{dcsr} and \eqref{pisr} is equivalent to Gaussian lower bounds and uniform Harnack inequalities for the heat equation $L - \p_t$. For the relevant statements we refer the reader to Theorems \ref{T:gb}  and \ref{T:H} below.



Another basic result which follows from Theorem \ref{T:main} is a generalized Liouville type theorem, see Theorem \ref{T:cm} below, stating that, for any given $N\in \mathbb N$, 
\begin{equation}\label{cm}
\text{dim}\ \mathcal H_N(\M,L) < \infty,
\end{equation}
where we have indicated with $\mathcal H_N(\M,L)$ the linear space  of $L$-harmonic functions on $\M$ with polynomial growth of order $\le N$ with respect to the distance $d$.

In closing we mention that the framework of the present paper is analogous to that of the work \cite{BG}, where two of us have used the generalized curvature-dimension inequality in Definition \ref{D:cdi} to establish various global properties such as:
\begin{itemize}
\item[(i)] An a priori Li-Yau gradient estimate for solutions of the heat equation $L - \p_t$ of the form $u(x,t) = P_t f(x)$, where $P_t=e^{tL} $ is the heat semigroup associated with $L$;
\item[(ii)] A scale invariant Harnack inequality for solutions of the heat equation of the form $u = P_t f$, with $f\ge 0$;
\item[(iii)] A Liouville type theorem for solutions of $Lf = 0$ on $\M$;
\item[(iv)] Off-diagonal upper bounds for the fundamental solution of $L-\p_t$;
\item[(v)] A Bonnet-Myers compactness theorem for the metric space $(M,d)$.
\end{itemize}
 

As for the ideas involved in the proof of Theorem \ref{T:main} we mention that our approach is purely analytical and it is exclusively based on some new entropy functional inequalities for the heat semigroup. 
Our central result in the proof of Theorem \ref{T:main} is a uniform H\"older estimate of the caloric measure associated with the diffusion operator $L$. Such estimate is contained in Theorem \ref{T:estimee-P-boule} below, and it states the existence of an absolute constant $A>0$, depending only the parameters in the inequality CD$(\rho_1,\rho_2,\kappa,d)$, such that for every $x \in \mathbb{M}$, and $r>0$,
\begin{equation}\label{Pb}
P_{Ar^2} (\mathbf{1}_{B(x,r) })(x) \ge \frac{1}{2}.
\end{equation}
Here, for a set $E\subset \M$, we have denoted by $\mathbf 1_E$ its indicator function. 
Once the crucial estimate \eqref{Pb} is obtained, with the help of the Harnack inequality
\begin{equation}\label{harnack:intro}
P_s f (x) \le  P_t f(y) \left(\frac{t}{s}\right)^{\frac{D}{2}} \exp\left(
\frac{D}{m} \frac{d(x,y)^2}{4(t-s)} \right), \quad s<t,
\end{equation}
that was proved in \cite{BG} (for an explanation of the parameter $D$ see \eqref{D} below), the proofs of \eqref{dcsr}, \eqref{pisr} become fairly standard, and they rely on a powerful circle of ideas that may be found in the literature. 

The proof of \eqref{Pb} which represents the main novel contribution of the present work is rather  technical. We mention that the main building block is a dimension dependent reverse logarithmic Sobolev inequality in Proposition \ref{P:sharper-reverse-log-sob} below. We stress here that, even in the Riemannian case, which is of course encompassed by the present paper, such estimates are new and lead to some delicate reverse Harnack inequalities which constitute the key ingredients in the proof of \eqref{Pb}. Still in connection with the Riemannian case, it is perhaps worth noting that, although as we have mentioned, in this setting the inequalities \eqref{gdc}, \eqref{gpi} are of course well-known, nonetheless our approach provides a new perspective based on a systematic use of the heat semigroup. The more PDE oriented reader might in fact find somewhat surprising that one can develop the whole local regularity starting from a global object such the heat semigroup. This in a sense reverses the way one normally proceeds, starting from local solutions. 

Finally, we mention that in the recent paper \cite{BG2} two of us have obtained a purely analytical proof of \eqref{Pb} for complete Riemannian manifolds with Ric $\ge 0$. The approach in that paper, which is based on a  functional inequality much simpler than the one found in this paper, is completely different from that of Theorem \ref{T:estimee-P-boule} below and cannot be adapted to the non-Riemannian setting of the present paper. 

\section{Reverse logarithmic Sobolev inequalities for the heat semigroup}

\subsection{Framework}\label{SS:frame}

Hereafter in this paper, $\bM$ will be a $C^\infty$ connected manifold endowed with a smooth measure $\mu$ and a second-order diffusion operator $L$ on $\M$ with real coefficients,  locally subelliptic, satisfying $L1=0$ and 
\begin{equation*}
\int_\bM f L g d\mu=\int_\bM g Lf d\mu,\ \ \ \ \ \ \int_\bM f L f d\mu \le 0,
\end{equation*}
for every $f , g \in C^ \infty_0(\bM)$. 
We indicate with $\Gamma(f)$ the quadratic differential form defined by \eqref{gamma} and  denote by $d(x,y)$ the canonical distance associated with $L$ as in \eqref{di} in the introduction. 

There is another useful distance on $\M$ which in fact coincides with $d(x,y)$. Such distance is based on the notion of subunit curve introduced by Fefferman and Phong in \cite{FP}, see also \cite{JSC2}. By a result in \cite{PS}, given any point $x\in \M$ there exists an open set $x\in U\subset \M$ in which the operator $L$ can be written as 
\begin{equation}\label{locrep}
L = - \sum_{i=1}^m X^*_i X_i,
\end{equation}
where  the vector fields $X_i$ have Lipschitz continuous coefficients in $U$, and $X_i^*$ indicates the formal adjoint of $X_i$ in $L^2(\M,d\mu)$. We remark that such local representation of $L$ is not unique. A tangent vector $v\in T_x\M$ is called subunit for $L$ at $x$ if   
$v = \sum_{i=1}^m a_i X_i(x)$, with $\sum_{i=1}^m a_i^2 \le 1$. It turns out that the notion of subunit vector for $L$ at $x$ does not depend on the local representation \eqref{locrep} of $L$. A Lipschitz path $\gamma:[0,T]\to \M$ is called subunit for $L$ if $\gamma'(t)$ is subunit for $L$ at $\gamma(t)$ for a.e. $t\in [0,T]$. We then define the subunit length of $\gamma$ as $\ell_s(\gamma) = T$. Given $x, y\in \M$, we indicate with 
\[
S(x,y) =\{\gamma:[0,T]\to \M\mid \gamma\ \text{is subunit for}\ L, \gamma(0) = x,\ \gamma(T) = y\}.
\]
We remark explicitly that the assumption (H.4) in this paper can be reformulated by saying that 
\begin{equation}\label{S}
S(x,y) \not= \varnothing,\ \ \ \ \ \text{for every}\ x, y\in \M.
\end{equation}
Now,  it is easy to verify that \eqref{S} implies that for any $x,y\in \M$ one has
\begin{equation}\label{ds}
d_s(x,y) = \inf\{\ell_s(\gamma)\mid \gamma\in S(x,y)\}<\infty,
\end{equation}
and therefore \eqref{ds} defines a true distance on $\M$ (once we have the finiteness of $d_s$ the other properties defining a distance are easily verified). Furthermore, in Lemma 5.43 in \cite{CKS} it is proved that
\begin{equation}\label{dds}
d(x,y) = d_s(x,y),\ \ \ x, y\in \mathbb M.
\end{equation}
Therefore, also $d$ is a true distance on $\M$ and, in view of \eqref{dds}, we can work indifferently with either one of the distances $d$
or $d_s$. 

In closing, we mention if $L$ is in the form $L = \sum_{i=1}^m X_i^2 + X_0 $, with vector fields which are $C^\infty$ and satisfying the so-called H\"ormander's finite rank condition on the Lie algebra, then the Theorem of Chow-Rashevsky guarantees the validity of (H.4). 
If moreover $L$ has real-analytic coefficients,  then thanks to Theorem 2.2 on p.107 in \cite{De} we know that $L$ is hypoelliptic if and only if it satisfies H\"ormander's finite rank condition. Therefore, in this situation, the hypoellipticity of $L$ would guarantee the validity of (H.4). For generalizations of the cited result in \cite{De} to more general hypoelliptic operators with real-analytic coefficients, see \cite{OR}.

\subsection{Preliminary results}\label{SS:prelim}

In what follows we collect some results from \cite{BG} which will be needed in this paper. In the framework of Section \ref{SS:frame}, $L$ is essentially self-adjoint on $C^ \infty_0(\bM)$.  Due to the hypoellipticity of $L$, the function $(t,x) \rightarrow P_t f(x)$ is
smooth on $ (0,\infty) \times \mathbb{M} $ and
\[ P_t f(x)  = \int_{\mathbb M} p(x,y,t) f(y) d\mu(y),\ \ \ f\in
C^\infty_0(\mathbb M),\] where $p(x,y,t) = p(y,x,t) > 0$ is the so-called heat
kernel associated to $P_t$.

Henceforth in this paper we denote
\[
C^\infty_b(\bM) = C^\infty(\bM) \cap L^\infty(\bM).
\]
For $\varepsilon>0$ we also denote by $\mathcal{A}_\varepsilon$ the set of functions $f \in C^\infty_b(\M)$ such that
\[
f=g+\varepsilon,
\]
for some $\varepsilon >0$ and some $g \in C^\infty_b(\M)$, $g \ge 0$, such that $g, \sqrt{\Gamma(g)}, \sqrt{\Gamma^Z(g)} \in L^2(\bM)$. As shown in \cite{BG}, this set is stable under the action of $P_t$, i.e.,  if  $f\in \mathcal{A}_\varepsilon$, then $P_t f \in  \mathcal{A}_\varepsilon$.

Let us fix $x\in \bM$ and $T>0$. Given a function $f\in \mathcal A_\varepsilon$, for $0\le t\le T$ we introduce the \emph{entropy functionals}
\[
\Phi_1 (t)=P_t \left( (P_{T-t} f) \Gamma (\ln P_{T-t}f) \right)(x),
\]
\[
\Phi_2 (t)=P_t \left( (P_{T-t} f) \Gamma^Z (\ln P_{T-t}f) \right)(x).
\]
For later use, we observe here that
\[
\frac{d}{dt} P_t \left( (P_{T-t} f) \ln P_{T-t}f  \right)(x)=P_t \left( (P_{T-t} f) \Gamma (\ln P_{T-t}f) \right)(x)=\Phi_1 (t),
\]
and thus, with the above notations,
\begin{equation}\label{antider}
\int_0^T\Phi_1(t) dt=P_T ( f\ln f )(x) -P_Tf(x) \ln P_T f(x).
\end{equation}
For the sake of brevity, we will often omit reference to the point $x\in \M$, and write for instance  $P_T f$ instead of $P_T f(x)$. This should cause no confusion in the reader.

The main source of the functional inequalities that will be studied in the present work is the following result that was proved in \cite{BG}:

\begin{theorem}\label{P:linearizedBL}
Let $a, b : [0,T]\to [0,\infty)$ and $\gamma : [0,T] \rightarrow \mathbb{R}$ be $C^1$ functions. For $\varepsilon>0$ and $f \in \mathcal{A}_\varepsilon$, we have 
\begin{align*}
& a(T) P_T \left(  f \Gamma (\ln f) \right) +b(T) P_T \left(  f \Gamma^Z (\ln f) \right)-a(0)(P_{T} f) \Gamma (\ln P_{T}f)-b(0)(P_{T} f) \Gamma^Z (\ln P_{T}f) \\
 \ge  & \int_0^T \left(a'+2\rho_1 a -2\kappa \frac{a^2}{b}-4\frac{a\gamma}{m} \right)\Phi_1 ds +\int_0^T(b'+2\rho_2 a) \Phi_2  ds
 \\
 & +\frac{4}{m}\int_0^T a\gamma ds\ LP_{T} f -\frac{2}{m} \int_0^T a\gamma^2 ds\ P_T f.
\end{align*}

\end{theorem}
Henceforth in this paper, we let
\begin{equation}\label{D}
D=\left( 1+ \frac{3\kappa}{2 \rho_2} \right) m.
\end{equation}

The following  scale invariant Harnack inequality for the heat kernel was also  proved in \cite{BG}.

\begin{proposition}\label{P:harnack}
Let $p(x,y,t)$ be the heat kernel on $\bM$. For every $x,y, z\in
\bM$ and every $0<s<t<\infty$ one has
\[
p(x,y,s) \le p(x,z,t) \left(\frac{t}{s}\right)^{\frac{D}{2}}
\exp\left(\frac{D}{m} \frac{d(y,z)^2}{4(t-s)} \right).
\]
\end{proposition}

A basic consequence of this Harnack inequality is the control of the  volume growth of balls centered at a given point.
\begin{proposition}\label{P:vg}
For every $x \in \bM$ and every $R_0 >0$ there is a constant
$C(m,\kappa,\rho_2)>0$ such that, 
\[
\mu \left( B(x,R)\right) \le \frac{C(m,\kappa,\rho_2)}{R_0^D
p(x,x,R_0^2)} R^D, \quad\ \  R \ge R_0.
\]
\end{proposition}

\begin{proof}
Fix $x\in \bM$ and
$t>0$. Applying Proposition \ref{P:harnack} to $ p(x,y,t)$
for every $y\in B(x,\sqrt t)$ we find
\[
p(x,x,t) \le  2^{\frac{D}{2}} e^{\frac{D}{4m}}\ p(x,y,2t) =
C(m,\kappa,\rho_2) p(x,y,2t).
\]
Integration over $B(x,\sqrt t)$ gives
\[
p(x,x,t)\mu(B(x,\sqrt t)) \le C(m,\kappa,\rho_2) \int_{B(x,\sqrt
t)}p(x,y,2t)d\mu(y) \le C(m,\kappa,\rho_2),
\]
where we have used $P_t1\le 1$. This gives the on-diagonal upper
bound
 \begin{equation}\label{odub1} p(x,x,t) \le
\frac{C(m,\kappa,\rho_2)}{\mu(B(x,\sqrt t))}.
\end{equation}

Let now $t>\tau >0$. Again,  from the Harnack inequality of Proposition
\ref{P:harnack}, we have
\[
p(x,x,t) \ge p(x,x,\tau) \left( \frac{\tau}{t} \right)^{\frac{D}{2}}.
\]
The inequality (\ref{odub1}) finally implies the desired conclusion by taking $t=R^2$ and $\tau=R_0^2$.
\end{proof}

\subsection{Reverse logarithmic Sobolev inequalities}\label{S:rlogsob}

In this section we derive some functional inequalities which will play a fundamental role in the proof of Theorem \ref{T:doubling} below.


\begin{proposition}\label{P:main}
Let $\varepsilon>0$ and $f \in \mathcal A_\varepsilon$ . For $x\in \M$,  $t, \tau >0$, and $C \in \mathbb{R}$, one has
\begin{align*}
& \frac{\tau}{\rho_2} P_t (f \Gamma(\ln f))(x) +\tau^2 P_t (f \Gamma^Z (\ln f))(x)
+ \frac{1}{\rho_2} \left( 1+\frac{2\kappa}{\rho_2} +\frac{4C}{m}\right) \bigg[P_t ( f\ln f )(x) - P_tf(x) \ln P_t f(x)\bigg]  \\
\ge &\frac{t+\tau}{\rho_2} P_t f(x) \Gamma(\ln P_t f)(x) +(t+\tau)^2 P_t f(x) \Gamma^Z(\ln P_t f)(x) + \frac{4Ct}{\rho_2 m}  LP_t f(x) -\frac{2C^2}{m \rho_2} \ln \left( 1+\frac{t}{\tau}\right) P_t f(x).
\end{align*}
\end{proposition}
\begin{proof}

Let $T, \tau >0$ be arbitrarily fixed. We apply Theorem \ref{P:linearizedBL} with $\rho_1=0$, in which we choose
\[
b(t)=(T+\tau-t)^2,\ \ a(t)=\frac{1}{\rho_2} (T+\tau-t),\ \ \gamma(t)=\frac{C}{T+\tau-t},\ \ 0\le t\le T.
\]
With such choices we obtain 
\begin{equation}\label{choices}
\begin{cases}
a' -2\kappa \frac{a^2}{b}-4\frac{a\gamma}{m} \equiv -\frac{1}{\rho_2}\left( 1+\frac{2\kappa}{\rho_2} +\frac{4C}{m}\right),
\\
b' + 2 \rho_2 a \equiv 0,
\\
\int_0^T \frac{4a\gamma}{m} =\frac{4CT}{\rho_2 m},
\\
\text{and}
\\
- \int_0^T  \frac{2a\gamma^2}{m} =-\frac{2C^2}{m \rho_2} \ln \left( 1+\frac{T}{\tau} \right).
\end{cases}
\end{equation}

Keeping \eqref{antider} in mind,  we  obtain the sought for conclusion with $T$ in place of $t$. The arbitrariness of $T>0$ finishes the proof.

\end{proof}

A first notable consequence of Proposition \ref{P:main} is the following reverse log-Sobolev inequality which was also observed in \cite{BB2}.

\begin{corollary}\label{P:reverse logsob}
Let $\varepsilon >0$ and $f \in \mathcal A_\varepsilon$. 
 For $x\in \M$,  $t>0$ one has
\[
t P_t f(x) \Gamma(\ln P_t f)(x)  +\rho_2 t^2  P_t f(x) \Gamma^Z(\ln P_t f)(x) \le \left(1+\frac{2\kappa}{\rho_2}\right) \big[P_t ( f\ln f )(x) -P_tf(x)\ln P_t f(x)\big].
\]
\end{corollary}

\begin{proof}
We first apply Proposition \ref{P:main} with $C=0$, and then we let $\tau \to 0^+$ in the resulting inequality.

\end{proof}

We may actually improve  Corollary \ref{P:reverse logsob} and obtain the following crucial dimension dependent reverse log-Sobolev inequality.

\begin{theorem}\label{P:sharper-reverse-log-sob}
Let $\varepsilon >0$ and $f \in \mathcal A_\varepsilon$,
 then for every $C \geq 0 $ and $\delta >0$, one has for $x\in \M$, 
  $t>0$,
  \begin{align}\label{P}
&   \frac{t}{\rho_2}  P_{t} f(x) \Gamma(\ln P_{t } f)(x) +t^2 P_{t} f(x) \Gamma^Z(\ln P_{t} f)(x) \\
& \le \frac{1}{\rho_2} \left( 1+\frac{2\kappa}{\rho_2} +\frac{4C}{m}\right) \big[P_{ t } (f \ln f)(x) - P_{t}f(x) \ln P_{ t} f(x) \big] 
\notag\\
& - \frac{4C}{\rho_2 m} \frac{t}{1+\delta} LP_{t } f(x)+\frac{2C^2}{m \rho_2} \ln \left( 1+\frac{1}{\delta}\right) P_{t} f(x).
\notag\end{align}
\end{theorem}

\begin{proof}
For $x\in \M$, $t, \tau >0$, we apply Proposition \ref{P:main} to the function $P_\tau f$ instead of $f$. Recalling that $P_t(P_\tau f) = P_{t+\tau} f$, we obtain, for all $C\in \R$, 
\begin{align}\label{Pttau}
& \frac{\tau}{\rho_2} P_t (P_\tau f \Gamma(\ln P_\tau f))(x) +\tau^2 P_t \big(P_\tau f \Gamma^Z (\ln P_\tau  f)\big)(x)
\\
&  +\frac{1}{\rho_2} \left( 1+\frac{2\kappa}{\rho_2} +\frac{4C}{m}\right) \big[P_t ( P_\tau f\ln P_\tau f )(x) - P_{t+\tau}f(x) \ln P_{t+\tau} f(x)\big]  \notag\\
& \ge \frac{t+\tau}{\rho_2}  P_{t+\tau} f(x) \Gamma(\ln P_{t+\tau} f)(x) +(t+\tau)^2 P_{t+\tau} f(x) \Gamma^Z(\ln P_{t+\tau} f)(x) 
\notag\\
&  + \frac{4C}{\rho_2 m} t LP_{t+\tau} f(x) -\frac{2C^2}{m \rho_2} \ln \left( 1+\frac{t}{\tau}\right) P_{t+\tau} f(x).
\notag\end{align}
Invoking Proposition \ref{P:reverse logsob} we now find for every $x\in \M$, $\tau >0$,
\[
\tau P_\tau f(x) \Gamma(\ln P_\tau f)(x)  +\rho_2 \tau^2  P_\tau f(x) \Gamma^Z(\ln P_\tau f)(x) \le \left(1+\frac{2\kappa}{\rho_2}\right) \big[P_\tau( f\ln f )(x) -P_\tau f(x)\ln P_\tau f(x)\big].
\]
If we now apply $P_t$ to this inequality, we obtain
\[
\tau P_t (P_\tau f \Gamma(\ln P_\tau f))(x) +\rho_2 \tau^2 P_t (P_\tau f \Gamma^Z (\ln P_\tau  f))(x) \le \left(1+\frac{2\kappa}{\rho_2}\right) \big[P_{t+\tau} ( f\ln f )(x) -P_t(P_{\tau}f \ln P_{\tau} f)(x)\big].
\]
We use this inequality to bound from above the first two terms in the left-hand side of \eqref{Pttau}, obtaining
\begin{align*}
&\frac{1+\frac{2\kappa}{\rho_2} }{\rho_2 } P_{t+\tau}( f\ln f )(x)+\frac{4C}{\rho_2 m}P_t(P_{\tau}f \ln P_{\tau} f)(x) -\frac{1}{\rho_2} \left( 1+\frac{2\kappa}{\rho_2} +\frac{4C}{m}\right) P_{t+\tau}f(x) \ln P_{t+\tau} f(x)  \\
& \ge \frac{t+\tau}{\rho_2}  P_{t+\tau} f(x) \Gamma(\ln P_{t+\tau} f)(x) +(t+\tau)^2 P_{t+\tau} f(x) \Gamma^Z(\ln P_{t+\tau} f)(x) 
\\
& + \frac{4C}{\rho_2 m} t LP_{t+\tau} f(x) -\frac{2C^2}{m \rho_2} \ln \left( 1+\frac{t}{\tau}\right) P_{t+\tau} f(x).
\end{align*}
Consider the convex function $\Phi(s) = s\ln s$, $s>0$. Thanks to Jensen's inequality, we have for any $\tau >0$ and $x\in \M$
\[
\Phi(P_\tau f(x)) \le P_\tau(\Phi(f))(x),
\]
which we can rewrite
\[
P_\tau f(x) \ln P_\tau f(x) \le P_\tau(f \ln f)(x).
\]
For $C\geq 0$, applying $P_t$ to this inequality we find
\[
\frac{4C}{\rho_2 m}P_t(P_{\tau}f \ln P_{\tau} f)(x) \le  \frac{4C}{\rho_2 m}P_{t+\tau} (f \ln f)(x).
\]
We therefore conclude, for $C\geq 0$,
\begin{align*}
&\frac{1}{\rho_2} \left( 1+\frac{2\kappa}{\rho_2} +\frac{4C}{m}\right) \big[P_{t+\tau} (f \ln f)(x) - P_{t+\tau}f(x) \ln P_{t+\tau} f(x) \big] \\
& \ge \frac{t+\tau}{\rho_2} P_{t+\tau} f(x) \Gamma(\ln P_{t+\tau} f)(x) +(t+\tau)^2 P_{t+\tau} f(x) \Gamma^Z(\ln P_{t+\tau} f)(x)
\\
& + \frac{4C}{\rho_2 m} t LP_{t+\tau} f(x) -\frac{2C^2}{m \rho_2} \ln \left( 1+\frac{t}{\tau}\right) P_{t+\tau} f(x).
\end{align*}
If in the latter inequality we now choose $\tau =\delta t$, we find:
\begin{align*}
&\frac{1}{\rho_2} \left( 1+\frac{2\kappa}{\rho_2} +\frac{4C}{m}\right) \big[P_{t+\delta t } (f \ln f)(x) - P_{t+\delta t}f(x) \ln P_{t+\delta t} f(x) \big] \\
& \ge \frac{t+\delta t}{\rho_2} P_{t+\delta t} f(x) \Gamma(\ln P_{t+\delta t } f)(x) +(t+\delta t )^2 P_{t+\delta t} f(x) \Gamma^Z(\ln P_{t+\delta t} f)(x) 
\\
& + \frac{4C}{\rho_2 m} t LP_{t+\delta t} f(x) -\frac{2C^2}{m \rho_2} \ln \left( 1+\frac{1}{\delta}\right) P_{t+\delta t} f(x).
\end{align*}
Changing $(1+\delta)t$ into $t$ in the latter inequality, we finally conclude:
\begin{align*}
&   \frac{t}{\rho_2}  P_{t} f(x) \Gamma(\ln P_{t } f)(x) +t^2 P_{t} f(x) \Gamma^Z(\ln P_{t} f)(x) \\
& \le \frac{1}{\rho_2} \left( 1+\frac{2\kappa}{\rho_2} +\frac{4C}{m}\right) \big[P_{ t }(f \ln f)(x) - P_{t}f(x) \ln P_{ t} f(x)\big]
\\
& -\frac{4C}{\rho_2 m} \frac{t}{1+\delta} LP_{t }f(x)+\frac{2C^2}{m \rho_2} \ln \left( 1+\frac{1}{\delta}\right) P_{t} f(x).
\end{align*}
This gives the desired conclusion \eqref{P}.

\end{proof}

\section{Volume doubling property}\label{S:dc}

 Our principal objective of this section is proving the following result.

\begin{theorem}[Global doubling property]\label{T:doubling}
The metric measure space $(\mathbb{M}, d , \mu)$ satisfies the global volume doubling property. More precisely, there exists a constant $C_1 = C_1(\rho_1,\rho_2,\kappa,d)>0$ such that for every $x \in \mathbb{M}$ and every $r>0$,
\[
\mu(B(x,2r)) \le C_1 \mu(B(x,r)).
\]
\end{theorem}

\subsection{Small time asymptotics}

As a first step, we prove a small time asymptotics result interesting in itself. In what follows for a given set $A\subset \M$ we will denote by $\mathbf 1_A$ its indicator   function.

\begin{proposition}\label{P:Hino}
Given $x\in \M$ and $r>0$, let $f=\mathbf{1}_{B(x,r)^c }$. One has,
\[
\liminf_{s \to 0^+} (-s\ln{ P_s f}(x))\ge \frac{r^2}{4}.
\]
\end{proposition}

\begin{proof}To prove the proposition it will suffice to show that
\[
\limsup_{t \to 0^+} (t \ln{ P_t f}(x))\le -\frac{r^2}{4}.
\]
Let $0<\varepsilon<r$. By the Harnack inequality of Proposition \ref{P:harnack} and the symmetry of the heat kernel, we have for $y \in \bM$ and $z \in B(x,\varepsilon)$,
\[
p(x,y,t)\le p(z,y,(1+\varepsilon)t)2^{D/m} e^{\frac{D\varepsilon}{4 m t}}.
\]
Therefore, multiplying the above inequality by $f(y)=\mathbf{1}_{B(x,r)^c } (y)$ and then integrating with respect to $y$, we obtain
\[
P_tf(x)\le (P_{(1+\varepsilon)t}f)(z)2^{D/m} e^{\frac{D\varepsilon}{4 m t}}.
\]
By integrating now with respect to $z \in B(x,\varepsilon)$, we get
\[
P_tf(x)\le\frac{2^{D/m} e^{\frac{D\varepsilon}{4 m t}}}{\mu(B(x,\varepsilon))}  \int_\bM \mathbf{1}_{B(x,\varepsilon) }(z) (P_{(1+\varepsilon)t}f)(z) d\mu(z) .
\]
Now, from Theorem 1.1 in \cite{HiR} (for which normalization differs from us by a factor $1/2$ because he considers the semigroup $e^{tL/2}$), we obtain:
 \[
\lim_{t \to 0} t \ln  \int_\bM \mathbf{1}_{B(x,\varepsilon) }(z) (P_{(1+\varepsilon)t}f)(z) d\mu(z)=-\frac{(r-\varepsilon)^2}{4(1+\varepsilon)}.
 \]
 This yields therefore
 \[
\limsup_{t \to 0^+} (t \ln{ P_t f}(x))\le -\frac{(r-\varepsilon)^2}{4(1+\varepsilon)}+\frac{D\varepsilon}{4m}.
\]
We conclude by letting $\varepsilon \to 0$. 
\end{proof}

\subsection{Reverse Harnack inequalities}\label{S:reverse}

As a second step toward the proof of Theorem \ref{T:doubling} we investigate some of the consequences of the reverse log-Sobolev inequality in Proposition \ref{P:sharper-reverse-log-sob} for functions $f$ such that $0\le f \le 1$ (later, we will apply this to indicator functions).


\begin{proposition}\label{prop-reverse-harnack}
Let $\varepsilon>0$,  $f\in \mathcal{A}_\varepsilon$, $\varepsilon \le f \le 1$, and consider the function 
$u(x,t)=\sqrt{-\ln P_t f(x)}$.
Then, with the convention that $\frac{1}{0}=+\infty$, we have 
\[
2t u_t +u+\left( 1+\sqrt{\frac{D^*}{2} } \right) u^{1/3}+\sqrt{\frac{D^*}{2} }u^{-1/3} \ge 0,
\]
where
\[
D^*=m \left( 1+\frac{2\kappa}{\rho_2} \right).
\]
\end{proposition}

\begin{proof}
Noting that we have
\[
\frac{t}{\rho_2}  P_{t} f(x) \Gamma(\ln P_{t } f)(x) +t^2 P_{t} f(x) \Gamma^Z(\ln P_{t} f)(x)  \ge 0,
\]
applying the inequality \eqref{P} in Theorem \ref{P:sharper-reverse-log-sob},  we obtain that for all $C \geq 0$,
\begin{align*}
\frac{m}{2} \left( 1+\frac{2\kappa}{\rho_2} +\frac{4C}{m}\right) P_{ t } (f \ln f)(x) - \frac{m}{2} \left( 1+\frac{2\kappa}{\rho_2} +\frac{4C}{m}\right) (P_{t}f)\ln P_{ t} f  - \frac{2Ct}{1+\delta} LP_{t } f+\frac{C^2}{\delta} P_{t} f \ge 0,
\end{align*}
where we used the fact that
\[
\ln \left( 1+\frac{1}{\delta} \right)\le \frac{1}{\delta}.
\]
On the other hand, the hypothesis $0\le f \le 1$ implies $f\ln f \le 0$. After dividing both sides of the above inequality by $P_t f $, we thus find
\begin{align*}
- \frac{m}{2} \left( 1+\frac{2\kappa}{\rho_2} +\frac{4C}{m}\right) \ln P_{ t} f  - \frac{2Ct}{1+\delta} \frac{LP_{t } f}{P_t f} +\frac{C^2}{\delta} \ge 0.
\end{align*}
Dividing both sides by $C>0$, this may be re-written 
\begin{align}\label{ineq3}
- \frac{D^*}{2C} \ln P_{ t} f  -2  \ln P_{ t} f  - \frac{2t}{1+\delta} \frac{LP_{t } f}{P_t f} +\frac{C}{\delta} \ge 0.
\end{align}
We now minimize the left-hand side of \eqref{ineq3} with respect to $C$. The minimum value is attained in
\[
C=\sqrt{-\frac{\delta D^*}{2} \ln P_t f}.
\]
Substituting this value in \eqref{ineq3}, we obtain
\[
\sqrt{\frac{2D^*}{\delta} }\sqrt{-\ln P_t f}  -2  \ln P_{ t} f  - \frac{2t}{1+\delta} \frac{LP_{t } f}{P_t f} \ge 0.
\]
With $u(x,t)=\sqrt{-\ln P_t f(x)}$, and noting that $u_t=-\frac{1}{2u} \frac{ L P_t f }{P_t f}$, we can re-write this inequality as follows,
\[
\sqrt{\frac{D^*}{2\delta} } +  u  + \frac{2t}{1+\delta} u_t  \ge 0,
\]
or equivalently,
\[
2t  u_t +u +\delta u +(1+\delta) \sqrt{\frac{D^*}{2\delta} } \ge 0.
\]
Finally, if we choose
\[
\delta=\frac{1}{u^{2/3}},
\]
we obtain the desired conclusion.

\end{proof}

We now introduce the function $g:(0,\infty)\to (0,\infty)$ defined by
\begin{align}\label{function_g}
g(v) = \frac{1}{v+\left( 1+\sqrt{\frac{D^*}{2} } \right) v^{1/3}+\sqrt{\frac{D^*}{2} }v^{-1/3}}.
\end{align}
One easily verifies that 
\[
\underset{v\to 0^+}{\lim} \sqrt{\frac{D^*}{2}}v^{-1/3} g(v) = 1,\ \ \ \ \ \underset{v\to \infty}{\lim} v g(v) = 1.
\]
These limit relations show that $g\in L^1(0,A)$ for every $A>0$, but $g\not\in L^1(0,\infty)$. Moreover, if we set \[
G(u) = \int_0^u g(v) dv,
\]
then $G'(u) = g(u) >0$, and thus $G:(0,\infty)\to (0,\infty)$ is invertible. Furthermore, as is seen from \eqref{function_g}, as $ u \to \infty$ we have
\begin{equation}\label{Gasym}
G(u)=\ln u +C_0+R(u),
\end{equation}
where $C_0$ is a constant and $\underset{u\to \infty}{\lim} R(u) = 0$.
At this point we notice that, in terms of the function $g(u)$, we can re-express the conclusion of Proposition \ref{prop-reverse-harnack} in the form
\[
2 t u_t + \frac{1}{g(u)} \ge 0.
\]
Keeping in mind that $g(u) = G'(u)$, we thus conclude
\begin{equation}\label{G'}
\frac{d G(u)}{d t} = G'(u) u_t \ge - \frac{1}{2t}.
\end{equation}
From this identity we now obtain the following basic result.

\begin{corollary}\label{C:ineq5}
Let $f\in L^\infty(\bM)$, $0 \le f \le 1$, then for any $x\in \M$ and $0<s < t$,
\[
G\left(\sqrt{-\ln P_t f(x)}\right) \ge  G\left(\sqrt{-\ln P_s f(x)}\right)  -\frac{1}{2} \ln \left( \frac{t}{s} \right).
\]
\end{corollary}

\begin{proof}
If $f \in \mathcal{A}_\varepsilon$ for some $\varepsilon$, the inequality is a straightforward consequence of the above results. In fact, keeping in mind that $u(x,t)=\sqrt{-\ln P_t f(x)}$, in order to reach the desired conclusion all we need to do is to integrate \eqref{G'} between $s$ and $t$. Consider now $f\in L^\infty(\bM)$, $0 \le f \le 1$. Let $h_n\in C^\infty_0(\M)$, with $0\leq h_n\leq 1$, and $h_n \nearrow 1$.  For $n \ge 0$, $\tau >0$ and $\varepsilon >0$, the function
\[
(1-\varepsilon)P_\tau( h_n f) +\varepsilon \in \mathcal{A}_\varepsilon. 
\]
Therefore,
\[
G\left(\sqrt{-\ln P_t ((1-\varepsilon)P_\tau( h_n f) +\varepsilon)(x)}\right) \ge  G\left(\sqrt{-\ln P_s ((1-\varepsilon)P_\tau( h_n f) +\varepsilon)(x)}\right)  -\frac{1}{2} \ln \left( \frac{t}{s} \right).
\]
Letting $\varepsilon \to 0$, $\tau \to 0$ and finally $n \to \infty$, we obtain the desired conclusion for $f$. This completes the proof.
\end{proof}

Combining Corollary \ref{C:ineq5} with Proposition \ref{P:Hino} we obtain the following key estimate.

\begin{proposition}\label{P:key}
Let $x\in \M$ and $r>0$ be arbitrarily fixed. There exists $C_0^*\in \R$, independent of $x$ and $r$, such that for any $t >0$,
\[
G\left(\sqrt{-\ln P_t \mathbf{1}_{B(x,r)^c }(x)}\right)  \ge  \ln \frac{r}{\sqrt{t}}  + C_0^*.
\]
\end{proposition}

\begin{proof}
Re-write the inequality claimed in Corollary \ref{C:ineq5} as follows
\[
G\left(\sqrt{-\ln P_t f(x)}\right) \ge  G(\sqrt{-\ln P_s f(x)})  + \ln \sqrt{s} - \ln \sqrt t,
\]
where we have presently let $f(y) = \mathbf 1_{B(x,r)^c}(y)$. Since for this function we have, from Proposition \ref{P:Hino}, $\underset{s\to 0^+}{\lim} (-\ln P_s f(x)) = \infty$, using \eqref{Gasym} we see that, for $s\to 0^+$,  the latter inequality is equivalent to
\[
G\left(\sqrt{-\ln P_t f(x)}\right) \ge \ln \sqrt{-s \ln P_s f(x)} - \ln \sqrt t + C_0 + R(\sqrt{-\ln P_s f(x)}).
\]
We now take the $\liminf$ as $s\to 0^+$ of both sides of this inequality. Applying Proposition \ref{P:Hino} we deduce
\[
G\left(\sqrt{-\ln P_t f(x)}\right) \ge \ln \frac{r}{2}- \ln \sqrt t + C_0 = \ln \frac{r}{\sqrt t} + C_0^*,
\]
where we have let $C_0^* = C_0 - \ln 2$. This establishes the desired conclusion.

\end{proof}

We are now in a position to prove  the central result in this paper.

\begin{theorem}\label{T:estimee-P-boule}
There exists a constant $A>0$ such that for every $x \in \mathbb{M}$, and $r>0$,
\[
P_{Ar^2} (\mathbf{1}_{B(x,r) })(x) \ge \frac{1}{2}.
\]
\end{theorem}

\begin{proof}
By the stochastic completeness of $\M$  we know that $P_t 1 = 1$. Therefore, 
\[
P_{Ar^2} (\mathbf{1}_{B(x,r)})(x) = 1 - P_{Ar^2} (\mathbf{1}_{B(x,r)^c})(x).
\]
We conclude that the desired estimate is equivalent to proving that there exists an absolute constant $A>0$ such that
\[
\sqrt{\ln 2} \le \sqrt{- \ln P_{Ar^2} (\mathbf{1}_{B(x,r)^c})(x)},
\]
or, equivalently,
\begin{equation}\label{final}
G\left(\sqrt{\ln 2}\right) \le G\left(\sqrt{- \ln P_{Ar^2} (\mathbf{1}_{B(x,r)^c})(x)}\right).
\end{equation}
At this point we invoke Proposition \ref{P:key}, which gives
\[
G\left(\sqrt{- \ln P_{Ar^2} (\mathbf{1}_{B(x,r)^c})(x)}\right) \ge \ln\left(\frac{1}{\sqrt A} + C_1\right).
\]
It is thus clear that, letting $A\to 0^+$, we can certainly achieve \eqref{final}, thus completing the proof.

\end{proof}

With Theorem \ref{T:estimee-P-boule}
in hand we can finally prove Theorem \ref{T:doubling}.

\begin{proof}[Proof of Theorem \ref{T:doubling}]
The argument which shows how to obtain Theorem \ref{T:doubling} from Theorem \ref{T:estimee-P-boule} was developed independently by Grigor'yan \cite{Gri} and by Saloff-Coste \cite{SC}, and it is by now well-known. However, since it is short for the sake of completeness in what follows we provide the relevant details. 

From the semigroup property and the symmetry of the heat kernel we
have for any $y\in \bM$ and $t>0$
\[ p(y,y,2t) = \int_\bM  p(y,z,t)^2 d\mu(z).
\]

Consider now a function $h\in C^\infty_0(\bM)$ such that $0\le h\le
1$, $h\equiv 1$ on $B(x,\sqrt{t}/2)$ and $h\equiv 0$ outside
$B(x,\sqrt t)$. We thus have
\begin{align*}
P_t h(y) & = \int_\bM p(y,z,t) h(z) d\mu(z) \le \left(\int_{\bM}
p(y,z,t)^2 d\mu(z)\right)^{\frac{1}{2}} \left(\int_\bM h(z)^2
d\mu(z)\right)^{\frac{1}{2}}
\\
& \le p(y,y,2t)^{\frac{1}{2}} \mu(B(x,\sqrt t))^{\frac{1}{2}}.
\end{align*}
If we take $y=x$, and $t =r^2$, we obtain
\begin{equation}\label{ine7}
P_{r^2} \left(\mathbf 1_{B(x,r/2)}\right)(x)^2 \le P_{r^2} h(x)^2 \leq
p(x,x,2r^2)\ \mu(B(x,r)).
\end{equation}
At this point we use Theorem \ref{T:estimee-P-boule}, which gives for some $0<A<1$, (the fact that we can choose $A<1$ is clear from the proof of Theorem \ref{T:estimee-P-boule})
\[
P_{Ar^2}(\mathbf 1_{B(x,r/2)})(x) \ge \frac{1}{2}, \ \ \ \ \ x\in \M, r>0.
\]
Combining this estimate with the Harnack inequality in Proposition \ref{P:harnack} and with \eqref{ine7}, we obtain the following 
on-diagonal lower bound
\begin{equation}\label{odlb}
p(x,x,2r^2) \ge \frac{C^*}{\mu(B(x,r))},\ \ \ \ \ x\in \bM,\ r>0.
\end{equation}
Applying Proposition \ref{P:harnack} 
we find for every $y\in B(x,\sqrt t)$, 
\[
p(x,x,t) \le  C p(x,y,2t).
\]
Integration over $B(x,\sqrt t)$ gives
\[
p(x,x,t)\mu(B(x,\sqrt t)) \le C \int_{B(x,\sqrt
t)}p(x,y,2t)d\mu(y) \le C,
\]
where we have used $P_t1\le 1$. Letting $t = r^2$, we obtain from this the on-diagonal upper
bound \begin{equation}\label{odub} p(x,x,r^2) \le
\frac{C}{\mu(B(x,r))}.
\end{equation}
Combining \eqref{odlb} with \eqref{odub} we finally obtain
\[
\mu(B(x,2r)) \le \frac{C}{p(x,x,4r^2)} \le \frac{C C'}{p(x,x,2r^2)}
\le C^{**} \mu(B(x,r)),
\]
where we have used once more Proposition \ref{P:harnack} (with $y=z=x$),
which gives
\[
\frac{p(x,x,2r^2)}{p(x,x,4r^2)}\le C',
\]
and we have let  $C^{**} = C C' (C^{*})^{-1}$.
This completes the proof.

\end{proof}

It is well-known that Theorem \ref{T:doubling} provides the following uniformity control at all scales.

\begin{theorem}\label{T:uniformity}
With $C_1$ being the constant in Theorem \ref{T:doubling}, let $Q = \log_2 C_1$. For any $x\in \bM$ and $r>0$
one has
\[
\mu(B(x,tr)) \ge C_1^{-1} t^{Q} \mu(B(x,r)),\ \ \ 0\le t\le 1.
\]
\end{theorem}

\section{Two-sided Gaussian bounds, Poincar\'e inequality and parabolic Harnack inequality}\label{S:ogb}

The purpose of this section is to establish some optimal two-sided bounds for the heat kernel $p(x,y,t)$ associated with the subelliptic operator  $L$. Such estimates are reminiscent of those obtained by Li and Yau for complete Riemannian manifolds having Ric $\ge 0$. 
As a consequence of the  two-sided Gaussian bound for the heat kernel, we will derive the  Poincar\'e inequality and the local parabolic Harnack inequality thanks to well-known results in the works \cite{FS}, \cite{KS3}, \cite{Gri} ,  \cite{SC}, \cite{St1}, \cite{St2}, \cite{St3}.

We assume, once again, that the assumptions (H.1)-(H.5) be satisfied, and that the generalized curvature-dimension inequality CD$(\rho_1,\rho_2,\kappa,m)$ hold, with $\rho_1\ge 0$. 
Here is our main result.

\begin{theorem}\label{T:gb}
For any $0<\ve <1$
there exists a constant $C(\ve) = C(m,\kappa,\rho_2,\ve)>0$, which tends
to $\infty$ as $\ve \to 0^+$, such that for every $x,y\in \bM$
and $t>0$ one has
\[
\frac{C(\ve)^{-1}}{\mu(B(x,\sqrt
t))} \exp
\left(-\frac{D d(x,y)^2}{m(4-\ve)t}\right)\le p(x,y,t)\le \frac{C(\ve)}{\mu(B(x,\sqrt
t))} \exp
\left(-\frac{d(x,y)^2}{(4+\ve)t}\right).
\]
\end{theorem}

\begin{proof}
We begin by establishing the lower bound. First, from Proposition \ref{P:harnack} we obtain for all $y \in \bM$, $t>0$, and every $0<\ve <1$,
\beas 
p(x,y,t)&\geq& p(x,x,\ve t)  \ve^\frac{D}{2} \exp\left( -\frac{D}{m}\frac{d(x,y)^2}{(4-\ve)t}\right).
\eeas
We thus need to estimate $p(x,x,\ve t)$ from below. But this has already been done in \eqref{odlb}. Choosing $r>0$ such that $2r^2 = \ve t$, we obtain from that estimate
\[
p(x,x,\ve t) \ge \frac{C^*}{\mu(B(x,\sqrt{\ve/2} \sqrt t))},\ \ \ \ \ x\in \bM,\ t>0.
\]
On the other hand, since $\sqrt{\ve/2}<1$, by the trivial inequality $\mu(B(x,\sqrt{\ve/2} \sqrt t)) \le \mu(B(x,\sqrt t))$, we conclude
\[
p(x,y,t) \geq \frac{C^*}{ \mu(B(x,\sqrt t))}  \ve^\frac{D}{2} \exp\left( -\frac{D}{m}\frac{d(x,y)^2}{(4-\ve)t}\right).
\]
This proves the Gaussian lower bound.


For the Gaussian upper bound, we first observe that the following upper bound is proved in \cite{BG}:
$$
p(x,y,t)\le \frac{C(m,\kappa,\rho_2,\ve')}{\mu(B(x,\sqrt
t))^{\frac{1}{2}} \mu(B(y,\sqrt
t))^{\frac{1}{2}}} \exp
\left(-\frac{d(x,y)^2}{(4+\ve')t}\right).
$$
At this point, by the triangle inequality and Theorem \ref{T:uniformity} we find.
\beas
\mu(B(x,\sqrt{ t})) &\leq& \mu(B(y,d(x,y)+\sqrt{ t}))\\
                 &\leq& C_1 \mu(B(y,\sqrt{ t})) \left(\frac {d(x,y)+\sqrt{ t}}{\sqrt t} \right)^Q.
\eeas
This gives
$$
\frac{1}{\mu(B(y,\sqrt{ t}))}\leq \frac{C_1}{\mu(B(x,\sqrt{ t}))} \left(\frac {d(x,y)}{\sqrt{ t}}+1 \right)^Q.
$$
Combining this with the above estimate we obtain
\[
p(x,y,t)\le \frac{C_1^{1/2}C(m,\kappa,\rho_2,\ve')}{\mu(B(x,\sqrt
t))}  \left(\frac {d(x,y)}{\sqrt{ t}}+1 \right)^{\frac{Q}{2}} \exp
\left(-\frac{d(x,y)^2}{(4+\ve')t}\right).
\]
If now $0<\ve<1$, it is clear that we can choose $0<\ve'<\ve$ such that 
\[
\frac{C_1^{1/2}C(m,\kappa,\rho_2,\ve')}{\mu(B(x,\sqrt
t))}  \left(\frac {d(x,y)}{\sqrt{ t}}+1 \right)^{\frac{Q}{2}} \exp
\left(-\frac{d(x,y)^2}{(4+\ve')t}\right) \le  \frac{C^*(m,\kappa,\rho_2,\ve)}{\mu(B(x,\sqrt
t))} \exp
\left(-\frac{d(x,y)^2}{(4+\ve)t}\right),
\]
where $C^*(m,\kappa,\rho_2,\ve)$ is a constant which tends to $\infty$ as $\ve \to 0^+$. The desired conclusion follows by suitably adjusting the values of both $\ve'$ and of the constant in the right-hand side of the estimate.

\end{proof}

With Theorems \ref{T:doubling} and \ref{T:gb} in hands, we can now appeal to the results in \cite{FS}, \cite{KS3}, \cite{Gri},  \cite{SC}, \cite{St1}, \cite{St2}, \cite{St3}, see also the books \cite{GSC}, \cite{Gri2}. More precisely, from the developments in these papers it is by now well-known that in the context of strictly regular local Dirichlet spaces we have the equivalence between:

\

\begin{enumerate}
\item[(1)] A two sided Gaussian bounds for the heat kernel (like in Theorem \ref{T:gb});
\item[(2)] The  conjunction of the volume doubling property and the Poincar\'e inequality (see Theorem \ref{T:P});
\item[(3)] The parabolic Harnack inequality (see Theorem \ref{T:H}).
\end{enumerate}

\

For uniformly parabolic equations in divergence form the equivalence between $(1)$ and $(3)$ was first proved in \cite{FS}. The fact that $(1)$ implies the volume doubling property is almost straightforward, the argument may be found in \cite{SCbook} p. 161. The fact that $(1)$ also implies the Poincar\'e inequality relies on a beautiful and general argument by Kusuoka and Stroock \cite{KS3}, pp. 434-435. The equivalence between $(2)$ and $(3)$ originates from \cite{Gri} and   \cite{SC} and has been worked out in the context of strictly local regular Dirichlet spaces in \cite{St3}. Finally, the fact that $(2)$ implies $(1)$ is also proven in \cite{St3}.

\

Thus, in our framework, thanks to Theorem \ref{T:gb}  we obtain the following weaker form of Poincar\'e inequality.
Of course we already know the volume doubling property since we proved it to obtain the Gaussian estimates.
 
\begin{theorem}\label{T:P}
There exists a constant $C = C(m,\kappa,\rho_2)>0$ such that for every $x\in \M, r>0$, and $f\in C^\infty(\M)$ one has
\[
\int_{B(x,r)} |f(y) - f_r|^2 d\mu(y) \le C r^2 \int_{B(x,2r)} \Gamma(f)(y) d\mu(y),
\]
where we have let $f_r = \frac{1}{\mu(B(x,r))} \int_{B(x,r)} f d\mu$.
\end{theorem}

Since thanks to Theorem \ref{T:doubling} the space $(\M,\mu,d)$, where $d = d(x,y)$ indicates the sub-Riemannian distance, is a space of homogeneous type, and it is also a \emph{length-space} in the sense of Gromov, arguing as in \cite{J} we now conclude with the following result.

\begin{corollary}\label{C:Pi}
There exists a constant $C^* = C^*(m,\kappa,\rho_2)>0$ such that for every $x\in \M, r>0$, and $f\in C^\infty(\M)$ one has
\[
\int_{B(x,r)} |f(y) - f_r|^2 d\mu(y) \le C^* r^2 \int_{B(x,r)} \Gamma(f)(y) d\mu(y).
\]
\end{corollary}

Furthermore, the following scale invariant Harnack inequality for local solutions holds.

\begin{theorem}\label{T:H}
If $u$ is a positive solution of the heat equation in a cylinder of the form $Q=(s,s+\alpha r^2) \times B(x,r)$ then
\begin{equation}\label{Ha}
\sup_{Q-} u\le C \inf_{Q+} u,
\end{equation}
where for some fixed $0 < \beta < \gamma <\delta<\alpha<\infty$ and $\eta \in (0,1)$,
\[
Q-=(s+\beta r^2,s+\gamma r^2)\times B(x,\eta r), Q+=(s+\delta r^2, s+\alpha r^2)\times B(x,\eta r).
\]
Here, the constant $C$ is independent of $x,r$ and $u$, but depends on the parameters $m, \kappa, \rho_2$, as well as on $\alpha, \beta, \gamma, \delta$ and $\eta$.
\end{theorem}

\section{$L$-harmonic functions with polynomial growth}\label{S:polgro}


In \cite{BG} the first and third named authors were able to establish a Yau type Liouville theorem stating that when $\M$ is complete, and the generalized curvature dimension inequality  \emph{CD}$(\rho_1,\rho_2,\kappa,m)$ holds for $\rho_1 \ge 0$, then there exist no bounded solutions of $Lf = 0$ on $\M$ besides the constants. Note that this result is weaker than Yau's original Riemannian result in \cite{Yau} since this author only assumes a one-side bound. However, as a consequence of Theorems \ref{T:doubling} and \ref{T:H} we can now remove such limitation and obtain the following complete sub-Riemannian analogue of Yau's Liouville theorem. 

\begin{theorem}\label{T:lio}
 There exist no positive solutions of $Lf = 0$ on $\M$ besides the constants.
\end{theorem}

In fact, we can now prove much more. In their celebrated work \cite{CM} Colding and Minicozzi obtained a complete resolution of Yau's famous conjecture that the space of harmonic functions with a fixed polynomial growth at infinity on an open manifold with Ric $\ge 0$ is finite dimensional.
A fundamental discovery in that paper is the fact that such property can be solely derived from the volume doubling condition and the Neumann-Poincar\'e inequality. In Theorem 8.1 in \cite{CM} the authors, assuming these two properties, present a generalization of their result to sub-Riemannian manifolds. However, at the time \cite{CM} was written the only application of such theorem that could be given was to Lie groups with polynomial volume growth, see Corollary 8.2 in that paper. 

If we combine Theorem \ref{T:doubling} and Corollary \ref{C:Pi} above with the cited Theorem 8.1 in \cite{CM}, we can considerably broaden the scope of Colding and Minicozzi's result and generalize it to the geometric framework covered by the present paper. We obtain in fact the following generalization of Yau's conjecture. Given a fixed base point $x_0\in \M$, and a number $N\in \mathbb N$, we will indicate with $\mathcal H_N(\M,L)$ the linear space of all solutions of $Lf = 0$ on $\M$ such that there exist a constant $C<\infty$ for which
\[
|f(x)| \le C(1 + d(x,x_0)^N),\ \ \ \ \ x\in \M.
\]

\begin{theorem}\label{T:cm}
For every $N\in \mathbb N$ one has: \emph{dim} $\mathcal H_N(\M,L) < \infty$.
\end{theorem}

\end{document}